\documentclass[reqno,a4paper,12pt]{amsart}

\usepackage[utf8]{inputenc}
\usepackage[T1]{fontenc}
\usepackage[english]{babel}
\usepackage[babel]{microtype}
\usepackage[a4paper, margin=1.1in]{geometry}
\usepackage[a4paper]{geometry}
\usepackage[dvipsnames]{xcolor}
\usepackage{amsmath,amssymb,amsthm}
\usepackage{amsrefs}
\usepackage{mathrsfs}
\usepackage{mathtools}
\usepackage{commath}
\usepackage{thmtools}
\usepackage{thm-restate}
\usepackage{bbm}
\usepackage{dsfont}
\usepackage{lmodern}
\usepackage{fixcmex}
\usepackage{cases}
\usepackage{enumitem}
\setlist[enumerate,1]{label=(\roman*)}
\usepackage{centernot}
\usepackage{booktabs}
\usepackage{multirow}
\usepackage{float}
\usepackage[pdfborderstyle={/S/U/W 0},unicode]{hyperref}
\hypersetup{
    colorlinks=true,
    linkcolor=blue!85!black,
    citecolor=blue!85!black,
    urlcolor=blue!85!black,
}
\usepackage{cleveref}
\usepackage{etoolbox}

\declaretheoremstyle[
  headfont=\normalfont\bfseries,
  bodyfont=\normalfont,
]{noital}

\declaretheorem[style=plain,numberwithin=section,name=Theorem]{theorem}
\declaretheorem[style=plain,sibling=theorem,name=Proposition]{proposition}
\declaretheorem[style=plain,sibling=theorem,name=Lemma]{lemma}

\newcommand{\indef}[1]{\emph{#1}}
\newcommand{\defined}{\mathrel{\coloneqq}}

\DeclarePairedDelimiter{\p}{\lparen}{\rparen}

\renewcommand{\leq}{\leqslant}

\renewcommand{\geq}{\geqslant}
\newcommand{\st}{\mathbin{\colon}}
\undef{\set}
\DeclarePairedDelimiter{\set}{\lbrace}{\rbrace}
\undef{\emptyset}
\newcommand{\emptyset}{\varnothing}
\DeclarePairedDelimiter{\card}{\lvert}{\rvert}

\DeclareMathOperator{\supp}{supp}

\newcommand{\RR}{\mathbb{R}}
\newcommand{\cS}{\mathcal{S}}
\newcommand{\bfx}{\mathbf{x}}
\newcommand{\bfy}{\mathbf{y}}
\newcommand{\bfz}{\mathbf{z}}

\begin{document}

\title{Localised graph Maclaurin inequalities}

\author{Lucas Aragão}
\address{IMPA, Estrada Dona Castorina 110, Jardim Botânico, Rio de Janeiro, RJ, Brasil}
\email{l.aragao@impa.br}

\author{Victor Souza}
\address{Department of Pure Mathematics and Mathematical Statistics, University of Cambridge, Cambridge, United Kingdom}
\email{vss28@cam.ac.uk}

\begin{abstract}
The Maclaurin inequalities for graphs are a broad generalisation of the classical theorems of Turán and Zykov.
In a nutshell they provide an asymptotically sharp answer to the following question: what is the maximum number of cliques of size $q$ in a $K_{r+1}$-free graph with a given number of cliques of size $s$? 
We prove an extensions of the graph Maclaurin inequalities with a weight function that captures the local structure of the graph.
As a corollary, we settle a recent conjecture of Kirsch and Nir, which simultaneously encompass the previous localised results of Bradač, Malec and Tompkins and of Kirsch and Nir.
\end{abstract}

\maketitle

%----------------------------------------------------------%

\section{Introduction}
\label{sec:intro}

One of the foundational results in extremal graph theory is Turán's theorem~\cite{Turan1941-fi}, which states that a graph $G$ that is $K_{r+1}$-free cannot have more edges than a balanced complete $r$-partite graph.
Zykov~\cite{Zykov1949-qr} later showed that these graphs also maximise the number of copies of $K_q$ among $K_{r+1}$-free graphs.

Denote by $K_s(G)$ the set of $s$-cliques in $G$ and $k_s(G) = \card{K_s(G)}$.
If $G$ is $K_{r+1}$-free, Zykov's theorem gives
\begin{equation}
\label{eq:zykov}
    k_q(G) \leq \binom{r}{q} \p[\Big]{\frac{k_1(G)}{r}}^q.
\end{equation}
The \emph{graph Maclaurin inequalities} are a broad extension of \eqref{eq:zykov}.
Indeed, they state that if $G$ is $K_{r+1}$-free, then
\begin{equation}
\label{eq:khadzhiivanov}
    \frac{k_1(G)}{\binom{r}{1}}
    \geq \p[\bigg]{\frac{k_2(G)}{\binom{r}{2}}}^{1/2}
    \geq \dotsb
    \geq \p[\bigg]{\frac{k_r(G)}{\binom{r}{r}}}^{1/r}.
\end{equation}
While Khadzhiivanov~\cite{Khadzhiivanov1977-xg} was the first to prove this result, his original proof had a gap, later filled by Nikiforov~\cite{Nikiforov2006-us}.
This inequality was also rediscovered and reproved by Sós and Straus~\cite{Sos1982-as}, by Fisher and Ryan~\cite{Fisher1992-is} and by Petingi and Rodriguez~\cite{Petingi2000-qw}.

Using \eqref{eq:khadzhiivanov}, one can address the following question: for $s < q \leq r$, what is the maximum number of copies of $K_q$ that an $K_{r+1}$-free graph with a given number of copies of $K_s$ can have?
Turán's theorem gives the exact answer for $s=1$ and $q=2$ and Zykov's theorem for $s=1$ and $q \geq 2$.
Eckhoff~\cite{Eckhoff2004-fe} and Frohmander~\cite{Frohmader2008-yx} gave further exact results.
Inequality \eqref{eq:khadzhiivanov} is asymptotically sharp and gives the precise answer under certain divisibility conditions.
Our main result is an strengthening of \eqref{eq:khadzhiivanov}.

%--------------------------------------%
\begin{theorem}
\label{thm:localised-khadzhiivanov}
Given a graph $G$ and integers $1 \leq s \leq q$, we have
\begin{equation}
\label{eq:localised-khadzhiivanov}
    \sum_{I \in K_q(G)} \binom{\sigma(I)}{s}^{q/s} \binom{\sigma(I)}{q}^{-1} \leq k_s(G)^{q/s},
\end{equation}
where $\sigma(I)$ is  the size of the largest clique in $G$ containing $I$.
Moreover, equality holds only when the subgraph of $G$ induced on the set of vertices that belong to an $s$-clique is a complete multipartite graph with equal parts.
\end{theorem}
%--------------------------------------%

Indeed, \Cref{thm:localised-khadzhiivanov} generalises \eqref{eq:khadzhiivanov}.
To see this, note that if $G$ is $K_{r+1}$-free, then $\binom{r}{s}^{q/s} \binom{r}{q}^{-1} \leq \binom{\sigma(I)}{s}^{q/s} \binom{\sigma(I)}{q}^{-1}$ for all $I \in K_q(G)$, as the function $t \mapsto \binom{t}{s}^{q/s} / \binom{t}{q}$ is decreasing for $t \geq q$.

An important feature \Cref{thm:localised-khadzhiivanov} is the local nature of the function $\sigma(I)$.
This result fits in an ongoing enterprise to show similarly \indef{localised} versions of results in extremal combinatorics.
The case $s = 1$ and $q = 2$, a localised version of Turán's theorem, was proposed in 2022 by Balogh and Lidický in a Oberwolfach~\cite{Krivelevich2022-sa} problem session.
Soon after, this case was settled independently by Bradač~\cite{Bradac2022-uj} and by Malec and Tompkins~\cite{Malec2022-oc}.
The full case $s = 1$, a localised version of Zykov's theorem, was then proven analytically by Kirsch and Nir~\cite{Kirsch2023-iq}.
They also conjectured in~\cite{Kirsch2023-iq}*{Conjecture 6.1} the case $s = 2$, which we settle in greater generality.

To prove \Cref{thm:localised-khadzhiivanov}, we follow the strategy of Nikiforov and Khadzhiivanov, building upon the Motzkin-Straus~\cite{Motzkin1965-rx} analytical proof of Turán's theorem.
We now review some aspects of this analytical approach.

For $\bfx = (\bfx_v)_{v \in V} \in \RR^V$, write $\bfx \geq 0$ (or $\bfx > 0$) if $\bfx_v \geq 0$ (or $\bfx_v > 0$) for all $v \in V$.
For a set $I \subseteq V$, denote the product $\bfx_I \defined \prod_{v \in I} \bfx_v$.
Given a graph $G$ and an integer $s$, define the following homogeneous polynomial
\begin{equation*}
    h_{s,G}(\bfx) \defined \sum_{J \in K_s(G)} \bfx_{J}.
\end{equation*}
The following inequalities appear in the work of Khadzhiivanov~\cite{Khadzhiivanov1977-xg}.
If $G$ is a $K_{r+1}$-free graph and $\bfx \geq 0$, then
\begin{equation}
\label{eq:maclaurin}
    \frac{h_{1,G}(\bfx)}{\binom{r}{1}}
    \geq \p[\bigg]{\frac{h_{2,G}(\bfx)}{\binom{r}{2}}}^{1/2}
    \geq \dotsb
    \geq \p[\bigg]{\frac{h_{r,G}(\bfx)}{\binom{r}{r}}}^{1/r}.
\end{equation}
Applying \eqref{eq:maclaurin} with $\bfx = 1$ (i.e. $\bfx_v = 1$ for all $v \in V$), we recover \eqref{eq:khadzhiivanov}.
In the case that $G = K_n$, the functions $h_{s,K_n}$ are the elementary symmetric polynomials, and for $r = n$, the inequalities \eqref{eq:maclaurin} are the classical Maclaurin inequalities (see \cite{Hardy1988-eh}*{p. 52}).
For this reason, we refer to \eqref{eq:maclaurin} (and \eqref{eq:khadzhiivanov}) as a Maclaurin inequality for graphs.
Our main technical result is the following.

%--------------------------------------%
\begin{theorem}
\label{thm:localised-maclaurin}
Given a graph $G$ and $1 \leq s \leq q$ and define
\begin{equation*}
%\label{eq:definition-f}
    f_{s,q,G}(\bfx) \defined \sum_{I \in K_q(G)} \binom{\sigma(I)}{s}^{q/s} \binom{\sigma(I)}{q}^{-1}\bfx_I.
\end{equation*}
Then, for every $\bfx \geq 0$, we have
\begin{equation}
\label{eq:localised-maclaurin}
    f_{s,q,G}(\bfx) \leq h_{s,G}(\bfx)^{q/s}.
\end{equation}
Moreover, equality holds for $\bfx > 0$ only when the subgraph of $G$ induced on the set of vertices that belong to an $s$-clique is a complete $\ell$-partite graph with parts $V_1, \dotsc, V_\ell$, for some $\ell \geq q$, and $\sum_{v \in V_i} \bfx_v = \sum_{u \in V_j} \bfx_u$ for all $1 \leq i,j \leq \ell$.
\end{theorem}
%--------------------------------------%

To obtain \Cref{thm:localised-khadzhiivanov} from \Cref{thm:localised-maclaurin}, just take $\bfx = 1$.
At first glance \eqref{eq:localised-maclaurin} seems stronger than \eqref{eq:localised-khadzhiivanov}, but they are in fact equivalent, as we will see in \Cref{prop:sidorenko}.

%----------------------------------------------------------%

\section{Localised inequalities for clique counts}

For a graph $G = (V,E)$, the \indef{clique number} $\omega(G)$ is the size of its largest clique.
For a subset $S \subseteq V$, denote by $G[S]$ the subgraph of $G$ spanned by $S$.
If a subset $I \subseteq V$ spans a clique, we denote by $\sigma_G(I)$ the size of the largest clique in $G$ containing $I$.
We omit the subscript and write $\sigma(I)$ whenever $G$ is clear from context.

Recall that for $\bfx = (\bfx_v)_{v \in V} \in \RR^V$, write $\bfx \geq 0$ if $\bfx_v \geq 0$ for all $v \in V$ and $\bfx > 0$ analogously.
The support of $\bfx$ is the set $\supp \bfx \defined \set{v \in V \st \bfx_v \neq 0}$.
For a set $I \subseteq V$, recall that $\bfx_I \defined \prod_{v \in I} \bfx_v$.
For integers $1 \leq s \leq q$, consider the function
\begin{equation}
\label{eq:weight}
    f_{s,q,G}(\bfx) \defined \sum_{I \in K_q(G)} \rho_{s,q}(\sigma(I)) \bfx_{I},
\end{equation}
where $\rho_{s,q}$ is defined, for $t \geq s$, as
\begin{equation*}
    \rho_{s,q}(t) \defined \binom{t}{s}^{q/s} \binom{t}{q}^{-1}.
\end{equation*}
Note that in \eqref{eq:weight}, $\rho_{s,q}(t)$ is only evaluated for $t \geq q$.
It is important to note that in this range, $\rho_{s,q}(t)$ is a decreasing function of $t$.
Indeed, one can write $\rho_{s,q}^s$ as a product of $sq$ functions, each of which is non-increasing in $t \geq q$.
Recall from the introduction that
\begin{equation*}
    h_{s,G}(\bfx) \defined \sum_{J \in K_s(G)} \bfx_{J}.
\end{equation*}
Note that $f_{s,q,G}$ is homogeneous of degree $q$ and $h_{s,G}$ homogeneous of degree $s$.
Moreover, for $\bfx > 0$, if $h_{s,G}(\bfx) = 0$, then $G$ has no $s$-clique, and thus $f_{s,q,G}(\bfx) = 0$ as well.
Therefore, to show that $f_{s,q,G}(\bfx) \leq h_{s,G}(\bfx)^{q/s}$ for all $\bfx \geq 0$, it is enough to do so restricted to
\begin{equation*}
    \cS_{s,G} = \set[\Big]{\bfx \in \RR^V \st \bfx \geq 0  \,,\,  h_{s,G}(\bfx) = 1}.
\end{equation*}
The inequality \eqref{eq:localised-maclaurin} in \Cref{thm:localised-maclaurin} is then equivalent to the following proposition.

%--------------------------------------%
\begin{proposition}
\label{prop:localised-maclaurin-inequality}
Given a graph $G$ and $1 \leq s \leq q \leq \omega(G)$, we have
\begin{equation}
\label{eq:homogeneus}
    f_{s,q,G}(\bfx) \leq 1,
\end{equation}
for every $\bfx \in S_{s,G}$.
\end{proposition}
%--------------------------------------%

We defer the discussion of the case of equality in \Cref{thm:localised-maclaurin} to \Cref{sec:equality}.
The goal of this section is to prove \Cref{prop:localised-maclaurin-inequality}.
For convenience, we denote
\begin{equation}
\label{eq:constant}
    M_{s,q,G} \defined \sup_{\bfx \in S_{s,G}} f_{s,q,G}(\bfx).
\end{equation}
We note that if $k_q(G) > 0$, then $M_{s,q,G} > 0$.
For every graph $G$, the set $\cS_{1,G}$ is the standard simplex, which is compact.
For $s \geq 2$, the set is $\cS_{s,G}$ is closed and unbounded.
The following proposition is the gives the crucial structural information about the optimisation problem.

%--------------------------------------%
\begin{proposition}
\label{prop:clique}
Given a graph $G$ and $1 \leq s \leq q \leq \omega(G)$, the function $f_{s,q,G}(\bfx)$ attains its maximum at a point $\bfx \in \cS_{s,G}$ with $\supp \bfx$ being a clique in $G$.
\end{proposition}
%--------------------------------------%

We will now see that \Cref{prop:clique} quickly gives us \Cref{prop:localised-maclaurin-inequality}.

%--------------------------------------%
\begin{proof}[Proof of \Cref{prop:localised-maclaurin-inequality}]
%If $h_{s,G}(\bfx) = 0$, then $\bfx_J = 0$ for all $J \in K_s(G)$.
%This implies that $\bfx_I = 0$ for all $I \in K_q(G)$ and thus $f_{s,q,G}(\bfx) = 0$, so \eqref{eq:homogeneus} holds.
%Now assume $h_{s,G}(\bfx) \neq 0$ and choose a scalar $\alpha > 0$ such that $\alpha \bfx \in \cS_{s,G}$.
%By homogeneity, we have $h_{s,G}(\alpha \bfx) = \alpha^s h_{s,G}(\bfx)$, so $\alpha \defined 1/(h_{s,G}(\bfx))^{1/s}$.
%Therefore, 
%\begin{equation*}
%    f_{s,q,G}(\bfx) = \alpha^{-q} f_{s,q,G}(\alpha \bfx) = \p[\big]{h_{s,G}(\bfx)}^{q/s} f_{s,q,G}(\alpha \bfx).
%\end{equation*}
%Our goal now is to show that $f_{s,q,G}(\alpha \bfx) \leq 1$.
By \Cref{prop:clique}, there is $\bfx^\ast \in \cS_{s,G}$ such that $f_{s,q,G}(\bfx^\ast) = M_{s,q,G}$ and $\supp \bfx^\ast$ is a clique, let's say, a $K_R$.
Recall that $\rho_{s,q}$ is decreasing, so
\begin{align*}
    f_{s,q,G}(\bfx^\ast)
    &\leq \sum_{I \in K_q(K_R)}\rho_{s,q}(\sigma_G(I)) \bfx^\ast_I
    \leq \rho_{s,q}(R) \sum_{I \in K_q(K_R)}\bfx^\ast_I \\
    &= \rho_{s,q}(R) h_{q,K_R}(\bfx^\ast).
\end{align*}
By Maclaurin's inequality \eqref{eq:maclaurin}, we have
\begin{equation*}
    h_{q,K_R}(\bfx^\ast) \leq \binom{R}{q} \p[\bigg]{\frac{h_{s,K_R}(\bfx^\ast)}{\tbinom{R}{s}} }^{q/s} = 1 / \rho_{s,q}(R).
\end{equation*}
Therefore, $f_{s,q,G}(\bfx^\ast) \leq 1$ as we wanted.
\end{proof}
%--------------------------------------%

With \Cref{prop:localised-maclaurin-inequality}, we can easily get the inequality \eqref{eq:localised-khadzhiivanov} in \Cref{thm:localised-khadzhiivanov} by setting $\bfx = 1$.
What is less clear to see is that \Cref{thm:localised-khadzhiivanov} if actually equivalent to \Cref{thm:localised-maclaurin}.
The idea of using combinatorial means to prove analytical inequalities can be traced back to Sidorenko~\cite{Sidorenko1987-rp}.

%--------------------------------------%
\begin{proposition}
\label{prop:sidorenko}
Inequality \eqref{eq:localised-khadzhiivanov} implies inequality \eqref{eq:localised-maclaurin}.
\end{proposition}
%--------------------------------------%
\begin{proof}
The inequality \eqref{eq:localised-khadzhiivanov} in \Cref{thm:localised-khadzhiivanov} says that $f_{s,q,G}(\bfx) \leq h_{s,G}(\bfx)$ for $\bfx = 1$.
To get the inequality for all $\bfx \geq 0$ note that since both $f_{s,q,G}$ and $h_{s,G}$ are continuous, it is enough to prove it for $\bfx$ with all coordinate rationals.
If some coordinate is $0$, we can remove the associated vertex.
Moreover, $f_{s,q,G}$ and $h_{s,G}^{q/s}$ are both homogeneous of the same degree, so we can rescale the coordinates of $\bfx$ to be all integers.

To go from integer coordinates to all 1's coordinates, we can consider the blowup of $G$.
Denote by $G_\bfx$ the graph obtained from $G$ by replacing each vertex $v$ with an independent set $U_v$ of size $\bfx_v$, and for every edge $uv \in E(G)$, replace the edge $uv$ with a complete bipartite graph with vertex classes $U_u$ and $U_v$.
Now observe that for every $J \in K_s(G)$ leads to the creation of $\bfx_J$ cliques in $G_\bfx$, thus
\begin{equation*}
    k_s(G_\bfx) = \sum_{J \in K_s(G)} \bfx_J = h_{s,G}(\bfx). 
\end{equation*}
Similarly, every $I \in K_q(G)$ leads to the creation of $\bfx_I$ cliques in $G_\bfx$ with the same value of $\sigma$, therefore
\begin{equation*}
    f_{s,q,G_\bfx}(1) = \sum_{I \in K_q(G)} \rho_{s,q}(\sigma_G(I)) \bfx_I = f_{s,q,G}(\bfx).
\end{equation*}
Finally, this gives
\begin{equation*}
    f_{s,q,G}(\bfx) = f_{s,q,G_\bfx}(1) \leq k_{s}(G_\bfx)^{q/s} = h_{s,G}(\bfx)^{q/s},
\end{equation*}
as claimed.
\end{proof}
%--------------------------------------%

The proof of \Cref{prop:clique} is divided in two steps.
The first step is a quite technical one, we must show that the supremum \eqref{eq:constant} is actually attained.
That is, we must show that there is $\bfx^\ast \in \cS_{s,G}$ such that $f_{s,q,G}(\bfx) \leq f_{s,q,G}(\bfx^\ast)$ for all $\bfx \in \cS_{s,G}$.
As pointed out before, the set $\cS_{1,G}$ is compact and the existence of $\bfx^\ast$ is trivial.
For $s \geq 2$, however, $\cS_{s,G}$ is closed and unbounded, so we must deal with this issue.
This is the content of \Cref{lem:attain}, the proof which we postpone to \Cref{sec:attain} as it is quite lengthy and not the highlight of the proof.

%--------------------------------------%
\begin{lemma}
\label{lem:attain}
Given a graph $G$ and $1 \leq s \leq q \leq \omega(G)$, the function $f_{s,q,G}(\bfx)$ attains a maximum with $\bfx \in \cS_{s,G}$.
\end{lemma}
%--------------------------------------%

The last ingredient of the proof of \Cref{prop:clique} is a symmetrisation argument.

%--------------------------------------%
\begin{lemma}
\label{lem:zykov}
Let $G$ be a graph and $1 \leq s \leq q \leq \omega(G)$.
Suppose that every vertex of $G$ is in an $s$-clique and that $f_{s,q,G}$ attains a maximum, restricted to $\cS_{s,G}$, at some point $\bfx > 0$.
If $u$ and $v$ are not adjacent, then there is $\bfy \in \cS_{s,G}$ such that $f_{s,q,G}(\bfy) = f_{s,q,G}(\bfx)$ and $\bfy_u = 0$.
\end{lemma}
%--------------------------------------%
\begin{proof}
Define $\bfy$ as
\begin{align*}
    \bfy_z \defined
    \begin{cases}
        \bfx_z & \text{if $z \neq u,v$}, \\
        \bfx_u + \xi_u & \text{if $z = u$}, \\
        \bfx_v + \xi_v & \text{if $z = v$},
    \end{cases}
\end{align*}
where $\xi_u$ and $\xi_v$ will be chosen later.
Observe that for any $w \in V$, we have
\begin{equation*}
    \frac{\partial h_{s,G}(\bfx)}{\partial \bfx_w} = \sum\limits_{\substack{J \in K_s(G) \\ w \in J}} \bfx_{J \setminus \set{w}}.
\end{equation*}
Thus, as $u$ and $v$ are not neighbours, we obtain
\begin{align}
\label{eq:h-derivative}
h_{s,G}(\bfy) = \xi_u \frac{\partial h_{s,G}(\bfx)}{\partial \bfx_u} + \xi_v \frac{\partial h_{s,G}(\bfx)}{\partial \bfx_v} + h_{s,G}(\bfx),
\end{align}
and similarly,
\begin{align}
\label{eq:f-derivative}
f_{s,q,G}(\bfy) = \xi_u \frac{\partial f_{s,q,G}(\bfx)}{\partial \bfx_u} + \xi_v \frac{\partial f_{s,q,G}(\bfx)}{\partial \bfx_v} + f_{s,q,G}(\bfx).
\end{align}

Note that $\partial h_{s,G}(\bfx) / \partial \bfx_w > 0$ for all $w \in V$ as $\bfx > 0$ and all vertices are in some $s$-clique.
We set
\begin{equation*}
    \xi_u = - \bfx_u, \qquad \xi_v = \bfx_u \frac{\partial h_{s,G}(\bfx) / \partial \bfx_u}{\partial h_{s,G}(\bfx) / \partial \bfx_v},
\end{equation*}
so $h_{s,G}(\bfy) = h_{s,G}(\bfx) = 1$ by \eqref{eq:h-derivative}.
In particular, $\bfy \in \cS_{s,G}$. 

By the Lagrange's method, as $f_{s,q,G}$ is maximised at $\bfx$ subject to $h_{s,G}(\bfx) = 1$, there is $\lambda \in \RR$ such that,
\begin{equation*}
    \frac{\partial f_{s,q,G}(\bfx)}{\partial \bfx_w} = \lambda \frac{\partial h_{s,G}(\bfx)}{\partial \bfx_w},
\end{equation*}
for all $w \in V$.
Together with \eqref{eq:f-derivative}, this implies that
\begin{equation*}
    f_{s,q,G}(\bfy) = \lambda\p[\Big]{\xi_u \frac{\partial h_{s,G}(\bfx)}{\partial \bfx_u} + \xi_v \frac{\partial h_{s,G}(\bfx)}{\partial \bfx_v}} + f_{s,q,G}(\bfx) = f_{s,q,G}(\bfx).
\end{equation*}
We are done as $\bfy_u = 0$.
\end{proof}
%--------------------------------------%

We are now ready to prove \Cref{prop:clique}, that is, to show that there is a maximum point $\bfx \in \cS_{s,G}$ whose support is a clique.

%--------------------------------------%
\begin{proof}[Proof of \Cref{prop:clique}]
Among the points $\bfx \in \cS_{s,G}$ with $f_{s,q,G}(\bfx) = M_{s,q,G}$, choose one with $\card{\supp \bfx}$ being minimal.
Such $\bfx$ exists from \Cref{lem:attain}.
Moreover, $\bfx$ is also a maxima restricted to the support $R = \supp \bfx$, that is, $f_{s,q,G[R]}(\bfx) = M_{s,q,G[R]}$.
Suppose that $G[R]$ is not a clique and let $u, v \in R$ be distinct vertices such that $uv \notin E(G)$.
Applying \Cref{lem:zykov} to $G[R]$, there is $\bfy \in \cS_{s,G[R]}$ such that $f_{s,q,G[R]}(\bfy) = M_{s,q,G[R]}$ and moreover, $\bfy_u = 0$.
This contradicts the minimality of $\card{\supp \bfx}$.
\end{proof}
%--------------------------------------%

As previously established, the inequalities in \Cref{thm:localised-khadzhiivanov} and \Cref{thm:localised-maclaurin} follow from \Cref{prop:clique}.

%----------------------------------------------------------%

\section{When equality holds}
\label{sec:equality}

Having established the inequalities \eqref{eq:localised-khadzhiivanov} in \Cref{thm:localised-khadzhiivanov} and \eqref{eq:localised-maclaurin} in \Cref{thm:localised-maclaurin}, we now determine when equality holds.
To do so, we must apply the symmetrisation argument of \Cref{lem:zykov} is a more careful way.
Moreover, we must use the fact that equality holds in the original Maclaurin's inequality \eqref{eq:maclaurin} for cliques only when all the coordinates are equal.

%--------------------------------------%
\begin{proposition}
\label{prop:localised-maclaurin-equality}
Let $G$ be a graph with clique number $\omega$ and $1 \leq s < q \leq \omega$. Let $U_s \subseteq V(G)$ be the set of vertices that are contained in an $s$-clique in $G$.
The equality
\begin{equation}
\label{eq:equality}
    f_{s,q,G}(\bfx) = h_{s,G}(\bfx)^{q/s},
\end{equation}
holds for $\bfx > 0$ if and only if the graph $G$ induced on $U_s$ is a complete $\omega$-partite graph with parts $V_1, \dotsc, V_\omega$ and $\sum_{v \in V_i} \bfx_v = \sum_{u \in V_j} \bfx_u$ for all $1 \leq i,j \leq \omega$.
\end{proposition}
%--------------------------------------%
\begin{proof}
Let $\bfx > 0$ be such \eqref{eq:equality} holds.
We may assume that $G = G[U_s]$, the values of $\bfx_v$ for $v \notin U_s$ do not interfere with the values of neither $f_{s,q,G}(\bfx)$ nor $h_{s,G}(\bfx)$.
As $\bfx > 0$, we also have $h_{s,G}(\bfx) > 0$, so by homogeneity of both sides of \eqref{eq:equality}, we may assume that $h_{s,G}(\bfx) = 1$.
By \Cref{thm:localised-maclaurin}, we then have $f_{s,q,G}(\bfx) = M_{s,q,G} = 1$.

Let $V_1, \dotsc, V_\ell$ be the vertices of the connected components of the complement of $G$.
We say that a clique $I \subseteq V(G)$ is \indef{canonical} if $\card{I \cap V_j} \leq 1$ for all $1 \leq j \leq \ell$.
Our goal is to show that $\sigma(I) = \ell$ for every canonical clique in $G$.

Let $U$ be a canonical $K_\ell$ in $G$.
Repeatedly applying \Cref{lem:zykov} over the non-edges in $V_i$, we find $\bfy \in \cS_{s,G}$ with $f_{s,q,G}(y) = f_{s,q,G}(x) = M_{s,q,G}$ and $\supp y = U$.
Indeed, we may do so as each $V_i$ is connected in the complement.
We obtain
\begin{align}
\label{eq:sigma}
    1 = f_{s,q,G}(\bfy) = \sum_{I \in K_q(U)} \rho_{s,q}(\sigma_G(I)) \bfy_I
    \leq \rho_{s,q}(\ell) h_{q,G[U]}(\bfy).
\end{align}
By Maclaurin's inequality \eqref{eq:maclaurin}, we have
\begin{align*}
    h_{q,G[U]}(\bfy) \leq \p[\big]{h_{s,G[U]}(\bfy)}^{q/s} \binom{\ell}{q} \binom{\ell}{s}^{-q/s}  = \rho_{s,q}(\ell)^{-1}.
\end{align*}
Therefore, we have equality in \eqref{eq:sigma}, which means that $\sigma_G(I) = \ell$ for all canonical copies of $K_q$ in $G[U]$.
Since a canonical $K_q$ in $G$ can be extended to a canonical $K_\ell$, we have that $\sigma_G(I) = \ell$ for every canonical $K_q$ in $G$.

We now claim that each $V_i$ is an independent set.
Indeed, suppose that there is an edge $uv \in G[V_i]$.
There must be a canonical clique $U$ in $G$ of size $\ell$ with $u \in U$.
Then $U \cup \set{v}$ must span a $K_{\ell+1}$ clique in $G$, which contradicts the fact that $\sigma_G(I) = \ell$ for every $K_q$ in $U$.
Therefore, $G$ is a complete multipartite graph, and thus, $\ell = \omega$.

Consider now the reduced graph $R$, which is a clique with vertex set $\set{1, \dotsc, \omega}$.
Let $\bfz \in \RR^R$ be defined as $\bfz_i = \sum_{v \in V_i} \bfx_v$.
For $W \subseteq V(R)$ denote by $V_{W} \defined \bigcup_{i \in W} V_i$.
Thus, if $W = \set{w_1, \dotsc, w_t}$, we have
\begin{equation*}
    \bfz_W = \prod_{w \in W} \p[\Big]{\sum_{v \in V_w}\bfx_v} = \sum\limits_{\substack{v_i \in V_{w_i} \\ i = 1, \dotsc, t}} \prod_{i = 1}^{t} \bfx_{v_i} = h_{t,V_W}(\bfx).
\end{equation*}
Since $\sigma(I) = \omega$ for every clique in $G$, we have
\begin{align*}
    h_{s,G}(\bfx) = \sum_{J \in K_s(G)} \bfx_J = \sum_{W \in K_s(R)} \sum_{J \in K_s(V_W)} \bfx_J  = \sum_{W \in K_s(R)} h_{s,V_W}(\bfx) = h_{s,R}(\bfz).
\end{align*}
In particular $h_{s,R}(\bfz) = 1$.
Similarly, we have
\begin{align*}
    1 = f_{s,q,G}(\bfx) = \rho_{s,q}(\omega) h_{q,G}(\bfx) = \rho_{s,q}(\omega) h_{q,R}(\bfz).
\end{align*}
By Maclaurin's inequality \eqref{eq:maclaurin},
\begin{align*}
    1 = \rho_{s,q}(\omega)h_{q,R}(\bfz) \leq \rho_{s,q}(\omega) \p[\big]{h_{s,R}(\bfz)}^{q/s} \binom{\omega}{q} \binom{\omega}{s}^{-q/s} = 1.
\end{align*}
As equality holds for Maclaurin inequality only when $\bfz_i$ are all equal, we are done.
The converse follows in the same way.
\end{proof}
%--------------------------------------%

Now, a proof of \Cref{lem:attain} is the only step missing for a complete proof of \Cref{thm:localised-khadzhiivanov} and \Cref{thm:localised-maclaurin}.

%----------------------------------------------------------%

\section{Attaining the maxima}
\label{sec:attain}

In this section, we give a proof of \Cref{lem:attain}.
Nikiforov~\cite{Nikiforov2006-us} noticed that Khadzhiivanov's~\cite{Khadzhiivanov1977-xg} proof of \eqref{eq:maclaurin} was incomplete as it assumed without proof that the supremum \eqref{eq:constant} was actually attained, which is not a triviality when $s \geq 2$.
We deal with this issue in essentially the same way that Nikiforov did.
Unfortunately, our proof is lengthy, for which we apologise.

%--------------------------------------%
\begin{proof}[Proof of \Cref{lem:attain}]
As $\cS_{1,G}$ is compact and $f_{1,q,G}$ attains a maximum in $\cS_{1,G}$, so assume $s \geq 2$.
If $s = q$, then $f_{s,q,G} = h_{s,G}$, so $f_{s,q,G}$ attains the maximum at any point $\bfx \in \cS_{s,G}$.
Assume $s < q$.

First note that $f_{s,q,G}$ is bounded on $\cS_{s,G}$.
Indeed, let $\bfx \in \cS_{s,G}$ and we give an uniform bound on $f_{s,q,G}(\bfx)$.
First observe that for $J \in K_s(G)$, we have $\bfx_J \leq h_{s,G}(\bfx) = 1$.
For any $t \geq s$, if $I \in K_t(G)$, the AM-GM inequality gives
\begin{align}
\label{eq:bound-cliques}
    \bfx_{I}
    = \p[\bigg]{\prod_{J \in K_s(I)} \bfx_J}^{1/\binom{t-1}{s-1}}
    \leq \p[\bigg]{\frac{\sum_{J \in K_s(I)} \bfx_J}{\binom{t}{s}}}^{\binom{t}{s}/\binom{t-1}{s-1}}
    \leq 1.
\end{align}
Applying this bound with $t = q$, and recalling that $\rho_{s,q}$ is decreasing, we obtain the bound $f_{s,q,G}(\bfx) \leq \rho_{s,q}(q) k_q(G) < \infty$ for all $\bfx \in \cS_{s,G}$.
In particular, $M_{s,q,G} < \infty$.

We prove this lemma by induction on $n$, the number of vertices of $G$.
Since $\omega(G) \geq q$, we may assume $n \geq q$.
For $n = q$, we can assume $G = K_{q}$.
For every $\bfx \in \cS_{s,G}$, the AM-GM inequality gives
\begin{align*}
    f_{s,q,G}(\bfx)
    &= \rho_{s,q}(q) \bfx_{V}
    \leq \rho_{s,q}(q)\p[\bigg]{\frac{\sum_{I \in K_s(K_q)}\bfx_I}{\binom{q}{s}}}^{\binom{q}{s}/\binom{q-1}{s-1}} = \rho_{s,q}(q) \binom{q}{s}^{-q/s} = 1.
\end{align*}
On the other hand, if $\bfy$ is defined as $\bfy_v = \binom{q}{s}^{-1/s}$ for all $v \in V$, then $\bfy \in \cS_{s,G}$ and
\begin{equation*}
    f_{s,q,G}(\bfy) = \rho_{s,q}(q) \bfy_{V} = \rho_{s,q}(q) \binom{q}{s}^{-q/s} = 1,
\end{equation*}
so the maximum is of $f_{s,q,G}$ is indeed attained in $\cS_{s,G}$,
and moreover $M_{s,q,K_q} = 1$.

Now, assume that the assertion holds for all graphs with $n-1$ vertices or fewer.
If $G$ contains a vertex $v$ not in any $K_s$ of $G$, then $x_v$ does not occur in $f_{s,q,G}$ nor in $h_{s,G}$.
That is to say, we have $f_{s,g,G}(\bfx) = f_{s,q,G - v}(\bfx')$ and $h_{s,G}(\bfx) = h_{s,G-v}(\bfx')$, where $\bfx' = (\bfx_u)_{u \in V(G - v)}$.
Therefore, the assertion holds for $G$ as it holds for $G - v$ by induction.
We now assume that every vertex of $G$ is contained in a copy of $K_s$.

Our goal is to show that there is $\bfy \in \cS_{s,G}$ with $f_{s,q,G}(\bfy) = M_{s,q,G}$.
Consider a sequence $\bfx^{(i)}$ in $\cS_{s,G}$ with $\lim_{i \to \infty}f_{s,q,G}(\bfx^{(i)}) = M_{s,q,G}$.
If for all $v \in V$, the sequence $\bfx^{(i)}_v$ is bounded, then $\bfx^{(i)}$ has an accumulation point $\bfy \in \cS_{s,G}$ and by continuity, $f_{s,q,G}(\bfy) = M_{s,q,G}$ as desired.

The remaining case is when there is a vertex $v \in V$, for which $\bfx^{(i)}_v$ in unbounded.
By our previous assumption, there is a clique $W \in K_s(G)$ with $v \in W$.
If there is some $c > 0$ such that  $\bfx^{(i)}_u > c$ for all $u \in W$, $u \neq v$ and all $i \geq 1$, then we have
\begin{equation*}
    1 \geq \bfx^{(i)}_J > c^{s-1} \bfx^{(i)}_v,
\end{equation*}
which contradicts the fact that $\bfx^{(i)}_v$ is unbounded.

Therefore, there is $u \in W$, $u \neq v$ for which $\liminf_{i \to \infty} \bfx^{(i)}_u = 0$.
We pass to a subsequence where $\bfx_u^{(i)} \to 0$, $\bfx_v^{(i)} \to \infty$ and $f_{s,q,G}(\bfx^{(i)}) \to M_{s,q,G}$ as $i \to \infty$.
Observe that
\begin{align}
    f_{s,q,G}(\bfx^{(i)}) 
    &= \sum_{I \in K_q(G), u \in I} \rho_{s,q}(\sigma_G(I)) \bfx^{(i)}_{I}
    + \sum_{I \in K_q(G), u \notin I} \rho_{s,q}(\sigma_G(I)) \bfx^{(i)}_{I} \nonumber\\
\label{eq:two-sums}
    &= \sum_{I \in K_{q}(G)} \rho_{s,q}(\sigma_G(I)) \bfx^{(i)}_{I \setminus \set{u}} \bfx^{(i)}_u
    + \sum_{I \in K_q(G-u)} \rho_{s,q}(\sigma_G(I)) \bfx^{(i)}_{I}.
\end{align}
To bound the first sum in \eqref{eq:two-sums}, recall from \eqref{eq:bound-cliques} that $\bfx^{(i)}_{I \setminus \set{u}} \leq 1$, and so
\begin{align}
\label{eq:first-bound}
    \sum_{I \in K_{q}(G)} \rho_{s,q}(\sigma_G(I)) \bfx^{(i)}_{I \setminus \set{u}} \bfx^{(i)}_u
    \leq \sum_{I \in K_{q}(G)} \rho_{s,q}(\sigma_G(I)) \bfx^{(i)}_u  \leq  \rho_{s,q}(q) k_q(G) \bfx^{(i)}_{u}.
\end{align}

For the second sum in \eqref{eq:two-sums}, observe that
\begin{align*}
    \sum_{I \in K_q(G-u)} \rho_{s,q}(\sigma_G(I)) \bfx^{(i)}_{I}
    \leq \sum_{I \in K_q(G-u)} \rho_{s,q}(\sigma_{G - u}(I)) \bfx^{(i)}_{I} = f_{s,q,G-u}(\tilde\bfx^{(i)}),
\end{align*}
where $\tilde \bfx^{(i)} = (\bfx^{(i)}_z)_{z \in V(G - u)}$.
The point $\tilde \bfx^{(i)}$ may not be in $\cS_{s,G-u}$, so we may need a rescaling.
To be precise, we need to rule out that $h_{s,G-u}(\tilde \bfx^{(i)}) = 0$.
Indeed, if that is the case, then  $\bfx_J^{(i)} = 0$ for all $J \in K_s(G-u)$, thus $\bfx_I^{(i)} = 0$ for all $I \in K_q(G-u)$.
In particular, $f_{s,q,G-u}$ is identically zero, so combining \eqref{eq:first-bound} with \eqref{eq:two-sums} and taking the limit $i \to \infty$, we have $M_{s,q,G} = 0$.
This implies that $k_q(G) = 0$, a contradiction.

We now assume that $h_{s,G-u}(\tilde \bfx^{(i)}) > 0$ and let $\alpha_i \defined 1/(h_{s,G-u}(\tilde\bfx^{(i)}))^{1/s}$, so that we have $\alpha_i \tilde\bfx^{(i)} \in \cS_{s,G-u}$.
Also notice that $h_{s,G-u}(\tilde \bfx^{(i)}) \leq h_{s,G}(\bfx^{(i)}) \leq 1$.
Therefore, we have
\begin{align}
\label{eq:second-bound}
    f_{s,q,G-u}(\tilde\bfx^{(i)})
    = \p[\big]{ h_{s,G-u}(\tilde\bfx^{(i)})}^{q/s} f_{s,q,G-u}(\alpha_i \tilde\bfx^{(i)})
    \leq f_{s,q,G-u}(\alpha_i \tilde\bfx^{(i)}).
\end{align}

By induction, there is a point $\tilde\bfy \in \cS_{s,G-u}$ at which $f_{s,q,G-u}$ attains the maximum, that is $f_{s,q,G-u}(\tilde\bfy) = M_{s,q,G-u}$.
Combining \eqref{eq:first-bound} and \eqref{eq:second-bound} back in \eqref{eq:two-sums}, we obtain
\begin{align}
\label{eq:prelimit}
    f_{s,q,G}(\bfx^{(i)}) 
    &\leq  \rho_{s,q}(q) k_q(G) \bfx^{(i)}_{u}
    + f_{s,q,G-u}(\tilde\bfy).
\end{align}
Define $\bfy \in \RR^V$ as $\bfy_z = \tilde\bfy_z$ for $z \in V(G-u)$ and $\bfy_u = 0$.
Note that $\bfy \in S_{s,G}$ as $h_{s,G}(\bfy) = h_{s,G-u}(\tilde\bfy)$.
Similarly, $f_{s,q,G-u}(\tilde\bfy) = f_{s,q,G}(\bfy) \leq M_{s,q,G}$.
Therefore, as $i \to \infty$, \eqref{eq:prelimit} implies
\begin{align*}
    M_{s,q,G} \leq f_{s,q,G}(\bfy) \leq M_{s,q,G}.
\end{align*}
Hence, the supremum is indeed attained in $\bfy \in \cS_{s,G}$.
\end{proof}
%--------------------------------------%

All the proofs are then complete.

%----------------------------------------------------------%

\section{Further directions}

In this paper, we have provided a so called localised version of the graph Maclaurin inequalities.
It would be of great interest to explore further which combinatorial results can be extended in this way.
Malec and Tompkins~\cite{Malec2022-oc} have provided, for instance, localised versions of Erdős-Gallai theorem, the Lubell–Yamamoto–Meshalkin-Bollobás inequality, the Erdős-Ko-Rado theorem and the Erdős-Szekeres theorem on monotone sequences.
Kirsch and Nir~\cite{Kirsch2023-iq} extended this list with several results in generalised Turán problems.
More importantly, we believe that these versions should have interesting applications that are yet to be discovered.

%----------------------------------------------------------%
 
\section{Acknowledgements}

The authors are grateful for the continued support of their respective supervisors, Rob Morris and Béla Bollobás.

%----------------------------------------------------------%

%----------------------------------------------------------%

\end{document}